\documentclass[12pt, letterpaper]{article}

\usepackage{fullpage}
\usepackage{amsfonts}
\usepackage{amsmath}
\usepackage{amsthm}
\usepackage{cite}
\usepackage{graphicx}
\usepackage[all]{xy}

\newcommand{\F}{\mathbb{F}}
\newcommand{\Z}{\mathbb{Z}}
\newcommand{\C}{\mathbb{C}}
\newcommand{\Q}{\mathbb{Q}}
\newcommand{\p}{\mathfrak{p}}
\newcommand{\q}{\mathfrak{q}}
\newcommand{\Imm}{\operatorname{Im}}
\newcommand{\Cl}{\operatorname{Cl}}
\newcommand{\Eva}{\operatorname{G}}
\newcommand{\GL}{\operatorname{GL}}
\newtheorem{theorem}{Theorem}
\newtheorem{lemma}{Lemma}
\newtheorem{corollary}{Corollary}
\newtheorem{remark}{Remark}

\begin{document}
\title{Solvable Number Field Extensions of Bounded Root Discriminant}
\author{Jonah Leshin}
\date{December 12, 2011}
\maketitle
\begin{abstract}
 Let $K$ be a number field and $d_K$ the absolute value
  of the discriminant of $K/\Q$. We consider the root discriminant
$d_L^{\frac{1}{[L:\Q]}}$ of extensions $L/K$. We show that for any $N>0$ and any positive
integer $n$, the set of
length $n$ solvable extensions of $K$ with root discriminant less than $N$ is
finite. The result is motivated by
the study of class field towers.
\end{abstract}

\section{Introduction}\label{intro}
\par Let $K$ be a number field with Hilbert class field $H_K$. Set
$K=K^{(0)}$ and define $K^{(i)}:=H_{K^{(i-1)}}$ for all $i\geq 1$. Let
$K^{(\infty)}=\cup_iK^{(i)}$. We say $F$ has finite class field tower
if $K^{(\infty)}$ is a finite extension of $\Q$, i.e., if the
$K^{(n)}$ stabilize for large $n$. If $K^{(\infty)}$ is infinite over
$\Q$, $K$ has infinite class field tower. The main theorem on class
field towers is that of Golod and Shafarevich, who proved in 1964 that
number fields with infinite class field tower exist \cite{MR0218331}. 
\par We define the root discriminant of $K$ to be
\begin{align*} 
rd(K):=d_K^{\frac{1}{[K:\Q]}}.
\end{align*}
 Given a tower of number
fields $L/K/F$, we have the following equality of ideals of $F$:
\begin{align}
 d_{L/F}=N_{K/F}(d_{L/K})(d_{K/F})^{[L:K]}, \label{(1)}
\end{align}
where $d_{L/F}$ denotes the relative discriminant. (With this
notation, $d_F$ denotes the absolute value of $d_{F/\Q}$). 
It follows from \eqref{(1)} that if $L$ is an extension of $K$, then
$rd(K)\leq rd(L)$, with equality if and only if $d_{L/K}=1$, i.e.,
$L/K$ is unramified at all finite primes. Thus if $K$ has infinite
class field tower, then all fields $K^{(i)}$ have the same
root discriminant; in particular, the set 
\begin{align*}
Z_{N,K}:=\{L:L/\Q \textrm{ finite}, L\supseteq K, \textrm{ and } rd(L)\leq N\}
\end{align*}
is infinite for $N\geq rd(K)$.
\par From discriminant bounds originally due to Odlyzko and others, there exist only finitely many number fields $K$ with $rd(K)\leq
\Omega:=4\pi e^{\gamma}\approx 22 $, where $\gamma$ is Euler's
constant, and this number can be improved to $2\Omega$ if we assume
the Generalized Riemann Hypothesis \cite{MR1890578}. It follows that a number field with infinite class field tower
must have root discriminant larger than $\Omega$. Martinet
has constructed a number field with infinite class field tower and
root discriminant $\approx 92.4$. Using tamely ramified class field
towers, Hajir and Maire gave an example of a number field $K$ with
$Z_{82.2,K}$ infinite \cite{MR1890578}. 
\par We will not be concerned with specific values of root
discriminants but rather with the following question: fix a number
field $K$ and an arbitrary (large) real number $N>0$. How can infinite
subsets of $Z_{N,K}$ arise? For example, fix a positive integer $n$. If $N$ were large enough, could the following set
be infinite:
\begin{align*}
\{L:L/\Q \textrm{ finite}, L/K \textrm{ solvable length $n$, }\textrm{and }
rd(L)\leq N\} \quad ?
\end{align*}
The answer to the question is no, and this is the main theorem of the
paper. 
\begin{theorem}\label{abelian}
Fix a number field $K$, a positive integer $n$, and a positive real number $N$. The set
\begin{align*}
Y_{n,N,K}:=\{L:L/\Q \textrm{ finite}, L/K \textrm{ solvable length n, } \textrm{and }
rd(L)\leq N\} 
\end{align*}
is finite. 
\end{theorem}
\begin{remark}
\begin{itemize}
\emph{
\item Taking $n=1$ gives finiteness for abelian extensions. The general solvable case follows by induction from the $n=1$
case, and the proof of the $n=1$ case occupies the bulk of the paper. We
set $Y_{N,K}:=Y_{1,N,K}$.
\item Rather than considering the root discriminant of extensions $L$ of $K$, we
could equivalently consider the quantity
$(N_{K/\Q}{d_{L/K}})^{1/[L:K]}$. This is evident by \eqref{(1)}.
\item Odlyzko mentions in \cite{MR1061762} that $Y_{N,\Q}$ is known to
  be finite for any $N$. }
\end{itemize} 
\end{remark}
From here on, all field extensions are assumed to be finite unless
otherwise stated. 

\section{Discriminants and Ramification Groups}\label{sec2}
Let $L/K$ be a Galois extension of local fields, with $K$ a finite
extension of $\Q_p$. In \cite{Zahl}, Hilbert gives the following formula for the
relative different $D_{L/K}$ in terms of the ramification groups $G_i$
of
$L/K$:
\begin{align*}
v_L(D_{L/K})=\sum_{i=0}^{\infty}(|G_i|-1), 
\end{align*}
 where $v_L$ denotes the normalized $\mathfrak{P}$-adic valuation of a fractional
 ideal of $O_L$, $\mathfrak{P}$ the unique maximal ideal of $O_L$. If
 $L/K$ is now a Galois extension of global fields with $\mathfrak{P}$
 a prime of $L$ lying above a prime $\mathfrak{p}$ of $K$, then
 \begin{align}
 v_{\mathfrak{p}} (d_{L/K})=gf \sum_{i=0}^{\infty}(|G_i|-1), \label{3}
\end{align}
where $G_i$ are the ramification groups of 
$\mathfrak{P}/\mathfrak{p}$, $g$ is the number of primes of $L$ above
$\mathfrak{p}$, and $f$ is the residue degree of
$\mathfrak{P}/\mathfrak{p}$. So for a Galois extension $K/\Q$, we obtain 
\begin{align*}
v_p(rd(K))=\frac{1}{|G_0|} \sum_{i=0}^{\infty}(|G_i|-1).
\end{align*} 
Note that if we define the root discriminant of a finite extension
$K_{\mathfrak{p}}$ of $\Q_p$ to be
\begin{align*}
 rd(K_{\mathfrak{p}})=p^{\frac{v_{\Q_p}(d_{K_\mathfrak{p}/\Q_p})}{[K_\mathfrak{p}:\Q_p]}},
\end{align*} then
  $p^{v_p(rd(K))}=rd(K_{\mathfrak{p}})$. 

\section{Proof of Theorem \ref{abelian}}
 Fix a number field $K$ and a real number $N>0$. Our first goal is to
 show that the set
\begin{align*}
X_{N,K}:=\{L:L/K \textrm{  abelian, } L/\Q \textrm{  Galois, and }
rd(L)\leq N \}
\end{align*}
is finite in the case when $K/\Q$ is Galois. 
\par  If $E/F$ is a Galois extension of number fields ramified at a
prime $\mathfrak{p}$ of $F$ with $e=e_E(\mathfrak{p})$, then by \eqref{3},
$\mathfrak{p}^{(e-1)fg}|d_{L/K}$. It follows that if $L$ is a number
field Galois over $\Q$ with $rd(L)\leq N$, then $L/\Q$ can not ramify at any rational
prime $p$ with $p>\sqrt{N}$. 
\par Let $S$ be the union of the real places
of $K$ and the set of primes of $K$ lying
above the rational primes $p$ with $p \leq
\sqrt{N}$. Suppose that $X_{N,K}$ is infinite. Then there exists an
increasing sequence of natural numbers $n_l$ such that
$[L_{n_l}:K]=n_l$ and $L_{n_l}\in X_{N,K}$. For a fixed positive integer
$m$, the maximal abelian
extension of $K$ of exponent $m$ that is unramified outside $S$ is
finite (see, e.g.,\cite{MR2514094}). Thus we may assume that $L_{n_l}/K$ is
\emph{cyclic} for each $l$. We deal with the following
two cases separately:

 \vspace{10pt}\noindent\textbf{Case I:}\enspace
 For every $\mathfrak{p} \in S$,
$\limsup\textrm{ } e_{L_{n_l}}(\mathfrak{p})< \infty$ (the $\limsup$ being
indexed by $l$). 
\newline
\par To ease notation, write $L_l$ for $L_{n_l}$. Write $S=\{\mathfrak{p}_j\}$. Let $\mathfrak{f}_l=\mathfrak{f}(L_l/K)$ be the conductor of $L_l/K$. As $l$ gets arbitrarily large, there exists
$j$ such that the power $a_j$ of $\mathfrak{p}_j$ dividing
$\mathfrak{f}_l$ gets
sufficiently large. This is because $L_l/K$ can only be ramified at
the primes in $S$, so its conductor is divisible by only these primes, which means
$L_l$ is contained in the ray class field of $K$ modulo the product of the infinite
real places of $K$ and a power $m_l$ of the product of
the finite primes of $K$ ramifying in $L_l/K$. As $[L_l:K]$ increases, the minimal such $m_l$ increases, which by definition of the ray
class field implies $a_j$ increases for some $j$. 
\par Set $\mathfrak{p}=\mathfrak{p}_j$ and $a=a_j$. Let $\mathfrak{P}_l$ be a prime of $L_l$ lying above
$\mathfrak{p}$. Write $\hat{L}_l$ for
$L_{l_{\mathfrak{P}_l}}$, and to keep notation consistent write
$\hat{K}$ for $K_{\mathfrak{p}}$. Define the group of $n$-units $U^n$ of
$\hat{K}$ to be the group of units
of $O_{\hat{K}}$ that are congruent to $1 \pmod{ \mathfrak{p}^n}$. The $\mathfrak{p}$-contribution to 
$\mathfrak{f}_l$ is the conductor of the abelian extension of local
fields $\hat{L}_l/\hat{K}$, which we denote by
$\hat{\mathfrak{f}}_l$. Let $\theta$ be the local reciprocity map
$\hat{K}^*\to \Eva(\hat{L}_l/\hat{K})$. Then $\hat{\mathfrak{f}}_l=\mathfrak{p}^{b_l}$,
where $b_l$ is by definition the smallest integer $n$ satisfying
$\theta(U^n)=1$. As $\theta$ maps $U^n$ onto the
$n^{th}$ upper ramification group $_lG^n$ of $_lG:=\Eva(\hat{L}_l/\hat{K})$, we
see that $_lG^{b_l-1}$ is non-trivial (in fact, $_lG^{b_l-1}$ is the last
non-trivial ramification group in the sense that
$_lG^{b_l-1+\epsilon}=1$ for any $\epsilon>0$). For a discussion of
the upper ramification groups, see \cite{MR554237}. 

\par By the previous two paragraphs, $l \to \infty$ implies $ a \to
\infty $, which implies $ b_l-1 \to \infty$. As $b_l-1 \to \infty$, the
largest integral index $c_l$ for which the $c_l$th lower
ramification group $_lG_{c_l}$ of $_lG$ is non-trivial tends to
$\infty$.  Now, let $p$ be the rational prime over which $\mathfrak{p}$
lies. We have $_lG\leq
\Eva(\hat{L}_l/\Q_p):=$$_l\Gamma$ and $_lG_n \leq$ $_l\Gamma_n$ for any $n$. (It follows from the
definition of the ramification groups that $_l\Gamma_{e(n+1)-1}=$$_lG_n$, where
$e=e(\mathfrak{p}/p)$). Therefore, as $l$ increases, the largest $n$ for which
$_l\Gamma_n$ is non-trivial increases as well. From
Section 2 we have that the exponent of $p$ in $rd(\hat{L}_l)$ is
\begin{align}
\frac{1}{|_l\Gamma_0|} \sum_{i=0}^{\infty}(|_l\Gamma_i|-1). \label{disc}
\end{align}
 The condition of Case I
means that $_l\Gamma_0$ can be bounded independently of $l$. So as
$l\to \infty$, \eqref{disc} does as well, which in turn implies that
$rd(\hat{L}_l)$, and
therefore $rd(L_l)$, tends to $\infty$, a contradiction. $\qed$

\par\vspace{10pt}\noindent\textbf{Case II:}\enspace: There exists $\mathfrak{p}\in S$ such that
  $\limsup\textrm{ } e_{L_{n_l}}(\mathfrak{p})= \infty$.  
\newline
\par Let $\mathfrak{m}=\mathfrak{m}_0\mathfrak{m}_{\infty}$ be a
modulus of $K$, where
$\mathfrak{m}_0$ is a product of finite primes and
$\mathfrak{m}_{\infty}$ is a product of $r$ real places of $K$. We have the following exact sequence from class field theory
\begin{align}
O^* \to (O/\mathfrak{m})^*\to \Cl_K^{\mathfrak{m}} \to \Cl_K \to 1,  \label{exact}
\end{align}
where $O_K^*$ are the units of the ring of integers $O=O_K$,
$\Cl_K$ is the ideal class group of $K$, $\Cl_K^{\mathfrak{m}}$ is
the ray class group of $K$ modulo $\mathfrak{m}$, and
$(O/\mathfrak{m})^*$ is defined to be $(O/\mathfrak{m}_0)^*\times
\{\pm 1\}^r$.

\par Let $p$ be the rational prime lying below $\p$. Because $L_l/\Q$ is assumed
to be Galois, the assumption of Case II implies that $\limsup\textrm{ } e_{L_{n_l}}(\mathfrak{p}')= \infty$ for every prime $\p'$ of $K$
with $\p'\cap \Z=(p)$. For any rational prime $q$ lying below a prime of $S$, define $\tilde{q}=\prod_{\q\in S,\q|q}\q$. For any
modulus $\mathfrak{m}$, let $R_{\mathfrak{m}}$ denote the ray class field of $K$ modulo
$\mathfrak{m}$. Let $\mathfrak{n}$ be the modulus
$\prod_{\mathfrak{q}\in
  S}\mathfrak{q}$ (note that whether a
prime of $K$ is contained in $S$ depends only on the rational prime
over which it lies). $L_l$ is contained in
$R_{\mathfrak{n}^{s(l)}}$ for some positive integer $s(l)$. Our
intermediate goal is to show
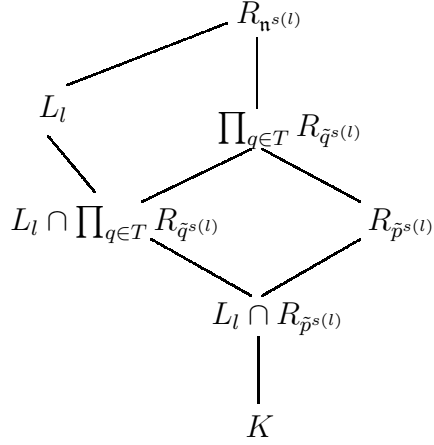
\begin{align*}
\limsup_{l \to \infty} [L_l\cap R_{\tilde{p}^{s(l)}}:K]=\infty.
\end{align*}
\par \textbf{Notation}: Suppose $\{E_l/F_l\}_l$ is a set of number field extensions
indexed by $l$ with $K\subseteq F_l$ for all $l$. We say $E_l/F_l$ is $(*)$ if the
ramification above $\mathfrak{p}$ of $E_l/F_l$ is bounded
independently of $l$. 

\par Because $\limsup e_{L_l}(\mathfrak{p})=\infty$, to prove our
intermediate goal, it suffices to show
that $L_l/L_l\cap R_{\tilde{p}^{s(l)}}$ is $(*)$. Let  $T$ be the set consisting of the
rational primes $\leq \sqrt{N}$ and the infinite real place of $\Q$, i.e., the set of rational primes below the primes
of $S$. Consider the field diagram below. It follows from the exact sequence
\eqref{exact} that the $p$-part of $[\prod_{q\in
  T}R_{\tilde{q}^{s(l)}}:R_{\tilde{p}^{s(l)}}]$, and therefore of
$[L_l\cap \prod_{q\in
  T}R_{\tilde{q}^{s(l)}}:L_l \cap R_{\tilde{p}^{s(l)}}]$, is bounded
independently of $l$. All the ramification above $\p$ in $\prod_{q \in
  T}R_{\tilde{q}^{s(l)}}/L_l \cap R_{\tilde{p}^{s(l)}}$ takes place in
$R_{\tilde{p}^{s(l)}}/L_l \cap R_{\tilde{p}^{s(l)}}$, and by \eqref{exact}, the
prime-to-$p$ part of $e_{R_{\tilde{p}^{s(l)}}}(\p)$ is bounded
independently of $l$. Therefore $L_l \cap \prod_{q\in
  T}R_{\tilde{q}^{s(l)}}/L_l \cap R_{\tilde{p}^{s(l)}}$ is $(*)$. Thus
it suffices to show that $L_l/L_l \cap \prod_{q\in
  T}R_{\tilde{q}^{s(l)}}$ is $(*)$. We accomplish this with the
following lemma. 

\begin{lemma}\label{l1}
Notation as above, $[R_{\mathfrak{n}^{s(l)}}:\prod_{q\in
  T}R_{\tilde{q}^{s(l)}}]$ is bounded independently of $l$. 

\end{lemma}

\begin{figure}[t]
\begin{center}
  \setlength{\unitlength}{2.5cm}
  \begin{picture}(2,2)(2,0)
    \put(3.05,0){$K$}
    \put(2.87,.6){$L_l\cap R_{\tilde{p}^{s(l)}}$}
    \put(1.8,1.1){$L_l\cap \prod_{q\in
    T}R_{\tilde{q}^{s(l)}}$}
    \put(3.7, 1.1){$ R_{\tilde{p}^{s(l)}}$}
    \put(2.9,1.6){$\prod_{q\in
    T}R_{\tilde{q}^{s(l)}}$}
    \put(3,2.2){$R_{\mathfrak{n}^{s(l)}}$} 
    \put(1.95,1.7) {$L_l$}
    \qbezier(3.12,.18)(3.12,.53)(3.12,.53)
    \qbezier(3.14, .75)(3.14,.75)(3.66,1.05)
    \qbezier(3.08, .75)(3.08, .75)(2.55, 1.05)
    \qbezier(2.5,1.25)(2.5, 1.25) (3.08, 1.53)
    \qbezier(3.66,1.25)(3.66,1.25)(3.14, 1.53)
    \qbezier(3.11, 1.75)(3.11,1.75)(3.11,2.12)
    \qbezier(2.25, 1.30)(2.25,1.30) (2,1.6)
    \qbezier(2.1, 1.87)(2.1,1.87)(2.95,2.2)  
    \end{picture}
\end{center}
     \end{figure}

\begin{proof}
The general case of the lemma will follow if we can show that
for any positive integers $a,b$:
\begin{align*}
[R_{\tilde{p}^a\tilde{q}^b}:R_{\tilde{p}^a}R_{\tilde{q}^b}]\leq \prod_{\p|p}(N\mathfrak{p}-1)\prod_{\q|q}(N\mathfrak{q}-1),
\end{align*}
where $q$ is a finite prime of $T$ distinct from $p$, and the products
are over primes of $K$ (if $q$
were the infinite place of $\Q$, one could take $2^{r_1}$ for the right-hand side of
the inequality, $r_1$ being the number of real
places of $K$). Define maps
\begin{align*}
\pi_{\tilde{p}\tilde{q}}:O^*\to (O/\tilde{p}\tilde{q})^*, \hspace{.2cm}
\pi_{\tilde{p}}:O^*\to (O/\tilde{p})^*, \hspace{.2cm}
 \pi_{\tilde{q}}:O^*\to (O/\tilde{q})^*.
\end{align*}

Let $H_K$ be the Hilbert class field of $K$. Looking at \eqref{exact} and using that $R_{\tilde{p}^a}\cap
R_{\tilde{q}^b}=H_K$ and that
$(O/\tilde{p}\tilde{q})^*\cong (O/\tilde{p})^*\times (O/\tilde{q})^*$, we see that
\begin{align*}
[R_{\tilde{p}^a\tilde{q}^b}:R_{\tilde{p}^a}R_{\tilde{q}^b}]=\frac{|\Imm (\pi_{\tilde{p}})||\Imm( \pi_{\tilde{q}})|}{|\Imm (\pi_{\tilde{p}\tilde{q}})|}.
\end{align*}
We have a decomposition
\begin{align*}
(O/\tilde{p}^a)^*\cong \prod_{\p|p}(O/\mathfrak{p})^* \times \bigl((1+\p)/(1+\p^a)\bigr)\cong 
\Big(\prod_{\p|p} \Z/(N\mathfrak{p}-1)\Big) \times P,
\end{align*}
where $P=\prod_{\p|p}\bigl((1+\p)/(1+\p^a)\bigr)$ is a $p$-group.
We have the analogous decomposition of $(O/\tilde{q}^b)^*$. Let $H=\Imm(\pi_{\tilde{p}\tilde{q}})$. Define the natural
projections 
\begin{align*}
\phi:(O/\tilde{p}\tilde{q})^*\to (O/\tilde{p})^*,\hspace{.2cm}
\psi:(O/\tilde{p}\tilde{q})^*\to (O/\tilde{q})^*,\\
\textrm{so that }
\Imm(\pi_{\tilde{p}})=\phi(H) \textrm{ and } \Imm(\pi_{\tilde{q}})=\psi(H).
\end{align*}
Since $|\phi(H)|$ and $|\psi(H)|$ both divide $|H|$ we obtain
\begin{displaymath}
\frac{|\phi(H)||\psi(H)|}{|H|}\leq \gcd \big(|\phi(H)|,|\psi(H)|\big)
\end{displaymath}
We obtain the result of the lemma by noting that

\begin{displaymath}
\gcd\big(|\phi(H)|,|\psi(H)|\big)\hspace{.1cm}\Big| \hspace{.1cm} \gcd\big(|(O/\tilde{p})^*|,|(O/\tilde{q})^*|\big)\hspace{.1cm}\Big|\hspace{.1cm}\prod_{\p|p}(N\p-1)\prod_{\q|q}(N\q-1).
\end{displaymath} 
\end{proof}

This completes the proof of our intermediate goal. With this in hand,
we now complete the proof that $X_{N,K}$ is finite under the
assumption of Case II. 

 Set $E_l=L_l \cap R_{\tilde{p}^{s(l)}}$, and let $\mathfrak{P}_l$ be a prime of of
$E_l$ lying above $\p$. Since $K/\Q$ is assumed to be Galois, the Galois closure of $E_l/\Q$ is an abelian
extension of $K$ ramified only above $p$, and it is contained in $L_l$
because $L_l/\Q$ is Galois. 
But $E_l$ is the maximal such field; 
 therefore $E_l/\Q$ is Galois. In proving our intermediate goal, we
 showed that $\limsup e_{E_l}(\p)=\infty$. Let $\hat{E}_l=E_{l_{\mathfrak{P}_l}}$ and
$\hat{K}=K_{\mathfrak{p}}$, and
consider the extension of local fields
$\hat{E}_{l}/\hat{K}$ with Galois group $_lG$. From the structure of
$(O/\tilde{p}^{s(l)})^*$ and the exact
sequence \eqref{exact} with $\mathfrak{m}=\tilde{p}^{s(l)}$, it follows
that there exists $C>0$ such that the prime-to-$p$ part of $|_lG|$ is
less than $C$ for all $l$. Let $F_{l}$ be the fixed field of the
inertia subgroup of $\p$ of $\Eva(E_l/K)$, so that $E_l/F_l$ is totally ramified above $\p$. Let $\mathfrak{Q}_l$ be the prime of $F_l$ below $\mathfrak{P}_l$
and set $\hat{F}_l=F_{l_{\mathfrak{Q}_l}}$ so that
$\hat{E}_{l}/\hat{F}_l$ is a totally ramified extension of local fields. Set
$_lH=\Eva(\hat{E}_l/\hat{F}_{l})$. If  $_lJ$ denotes the $p$-Sylow
subgroup of $_lH$, then $\limsup_{l\to \infty}|_lJ|=\infty$. Recall that
we are assuming that $L_l/K$ is cyclic, so all groups in sight are
cyclic. Let
$\hat{M}_l$ be the fixed field of $_lJ$, with maximal ideal $\mathfrak{r}_l$.

\begin{figure}[h]
\begin{center}
  \setlength{\unitlength}{2.5cm}
  \begin{picture}(3,2)(-.3,0)
    \put(3.05,0.05){$\p$}
    \put(3,.6){$\mathfrak{Q}_l$}
    \put(3.05,1.23){$\mathfrak{r}_l$}
    \put(3, 1.8){$\mathfrak{P}_l$}
    \qbezier(3.08,.21)(3.08,.18)(3.08,.53)
    \qbezier(3.08,.8)(3.08,.8)(3.08,1.15)
    \qbezier(3.08, 1.4)(3.08,1.4)(3.08,1.74)
  \end{picture}
 \begin{picture}(1,2)(4.9,0)
 \put(3,0){$\hat{K}$}
    \put(3,.6){$\hat{F}_l$}
    \put(3,1.2){$\hat{M}_l$}
    \put(3, 1.8){$\hat{E}_l$}
    \qbezier(2.96,.05)(2.15,.9)(2.96,1.85)
    \qbezier(3.22,1.84)(3.6,1.5)(3.24,1.25)
    \qbezier(3.22,1.88)(4.4,1.8)(3.20,.66)
    \put(2.25,.9){$_lG$}
    \put(3.455,1.53){$_lJ$}
    \put(3.55,.9){$_lH$}
    \qbezier(3.08,.21)(3.08,.18)(3.08,.53)
    \qbezier(3.08,.8)(3.08,.8)(3.08,1.15)
    \qbezier(3.08, 1.4)(3.08,1.4)(3.08,1.74)
\end{picture}
\end{center}
\end{figure}

\par Fix $l$ for the moment and put $\mathfrak{r}=\mathfrak{r}_l$. In this paragraph and the next, we summarize a few results from \cite{MR0220701} and \cite{MR554237}, and apply them to the situation at hand. For non-negative integers $n$, let $U^n$ be the $n$-units of
$\hat{M_l}$, as defined in the proof of Case I. For any real $v\geq 0$, define
$U^{v}=U^n$, where $n-1<v\leq n$. For all $v \geq 0$, the local reciprocity map
$\theta:\hat{M_l}^*\to$ $ _lJ$ maps $U^v$ onto the upper ramification
group $_lJ^v$. It follows that any jumps in the filtration $J^v$ must
occur at integral values of $v$ (i.e., if $J^{v+\epsilon}\neq
J^v$ for any $\epsilon>0$, then $v$ is an integer), and that for $v\geq 1$, the quotient $_lJ^v/_lJ^{v+1}$ of upper ramification groups is
isomorphic to a quotient of $U^v/U^{v+1}$. For integral $n\geq 1$, we have
isomorphisms $U^n/U^{n+1}\cong\mathfrak{r}^n/\mathfrak{r}^{n+1}\cong \F_{\hat{M}_l}^+$, the
last group being the additive group of the residue field of
$\hat{M}_l$, which is an elementary abelian $p$-group. Since $_lJ$ is
cyclic, we find that for $v\geq 1$, the quotients $_lJ^v/_lJ^{v+1}$
are either trivial or are cyclic groups of order $p$. 
\par For integral values of
$u$, the function $\phi$ satisfying $_lJ_u=$$_lJ^{\phi(u)}$
is given by 
\begin{align}
\phi(u)=\Bigg(\frac{1}{|_lJ_0|} \sum_{i=1}^u|_lJ_i|\Bigg)-1. \label{formula}
\end{align}
By \eqref{formula}, $\phi(n+1)-\phi(n)\leq 1$, from which it follows
that the $_lJ_n/_lJ_{n+1}$ are also either trivial or cyclic of order $p$.
Thus each subgroup of $_lJ$ occurs as $_lJ_n$ for some $n$. Because
jumps in the upper numbering occur at integral values, if $n$ is a non-negative integer with $_lJ_n\neq $$_lJ_{n+1}$, then
$\phi(n)\in \Z$. Using this fact along with \eqref{formula} and
the fact that every subgroup of $_lJ$ occurs as a ramification group,
we derive what is essentially the example on p.76 of \cite{MR554237}:

\begin{lemma}\label{lemma J}
For each $i$, $0\leq i \leq t$,  there exists $n_i$ such that
\begin{align*}
_lJ(i)={}_lJ_{n_i+1}={}_lJ_{n_i+2}=\cdots{}= _lJ_{n_i+p^i},
\end{align*} 
where $|_lJ|=p^t$ and $_lJ(i)$ denotes the unique subgroup of $_lJ$ of
order $p^{t-i}$. Less formally, the subgroup of order $p^{t-i}$ of $_lJ$
occurs as at least $p^i$ lower ramification groups of $_lJ$.
\end{lemma} 
\begin{remark}
\begin{itemize}
\emph{
\item
That jumps in the upper ramification groups of $_lJ^v$ occur only at
integral values of $v$ also follows from the Hasse-Arf Theorem, which
says that if $L/K$ is an abelian extension of local fields with perfect residue
fields, then jumps in the upper ramification groups $G^v$ of $L/K$ occur
only at integral values of $v$ \cite{MR554237}.
\item One can more directly obtain that $_lJ_n/_lJ_{n+1}$ is isomorphic to a
  subgroup of $U^n/U^{n+1}$ by considering the map $_lJ_n/_lJ_{n+1}
  \hookrightarrow U^n/U^{n+1}$ induced by $\sigma \mapsto
  \sigma(\beta)/\beta$ where $\beta$ is a uniformizer of $\hat{E}_l$ \cite{MR554237}.}
\end{itemize}
\end{remark}
\par We use the information just obtained about the ramification
groups of $_lJ$ to obtain bounds for $rd(L_l)$, which will complete
the proof of Case II. Since $E_l\subseteq L_l$, it suffices to show that
$\limsup_{l\to \infty}rd(E_l)=\infty$. By Section 2, it then suffices to show that $\limsup_{l \to \infty}rd(\hat{E}_l)=\infty$. The $n$th
ramification group of $\hat{E}_l/\hat{M}_l$ is a subgroup of the $n$th
ramification group of $\hat{E}_l/\Q_p$. By Section 2, it suffices to
show that 
\begin{align}
\limsup_{l \to \infty}
\frac{1}{|_l\Gamma_0|}\sum_{i=0}^{\infty}(|_l\Gamma_i|-1)=\infty, \label{Gamma}
\end{align}
where $_l\Gamma_n$ are the ramification groups of
$\hat{E}_l/\Q_p$. Let $_l\gamma_n$ denote the ramification groups of
$\hat{E}_l/\hat{M}_l$. It follows from Lemma \ref{lemma J} that
\begin{align}
\frac{1}{|_l\gamma_0|}\sum_{i=0}^{\infty}(|_l\gamma_i|-1)\geq t_l-1,  \label{t}
\end{align}
where $|_lJ|=p^{t_l}$. Note that $\limsup_{l \to \infty}t_l=\infty$. We
also have
\begin{align*}
|_l\Gamma_0|=e(\hat{E}_l/\hat{F}_l)e(\hat{F}_l/\hat{K})e(\hat{K}/\Q_p).
\end{align*}
The quantity $e(\hat{K}/\Q_p)$ is fixed and $e(\hat{F}_l/\hat{K})=1$.  By the definition of $_lJ$,
\begin{align*}
|_l\gamma_0|=|_lJ|\geq \frac{1}{C}|\Eva(\hat{E}_l/\hat{F}_l)|=\frac{1}{C}e(\hat{E}_l/\hat{F}_l).
\end{align*}
(Recall the definition of $C$ from the beginning of the discussion
after Lemma \ref{l1}). Thus we obtain the second inequality in:
\begin{align*}
\frac{1}{|_l\Gamma_0|}\sum_{i=0}^{\infty}(|_l\Gamma_i|-1)\geq
\frac{1}{|_l\Gamma_0|}\sum_{i=0}^{\infty}(|_l\gamma_i|-1) \geq
\frac{1}{C|_l\gamma_0| e(\hat{K}/\Q_p)}\sum_{i=0}^{\infty}(|_l\gamma_i|-1),
\end{align*}
which, using \eqref{t}, gives the desired result \eqref{Gamma},
completing Case II
. $\qed$

\par This completes the proof that $X_{N,K}$ is finite when $K/\Q$ is
Galois. We use this to show that $Y_{N,K}$ is finite. We initially maintain the assumption that $K/\Q$ is
Galois. Suppose $L \in Y_{N,K}$. We will show there is a constant $A:=A_{N;K}$,
independent of $L$, such that $rd(\tilde{L})< A_{N;K}$, where $\tilde{L}$
denotes the Galois closure of $L/\Q$. Suppose we have shown this to be
true. The field $\tilde{L}$ is a compositum of abelian
extensions of $K$, thus abelian over $K$. The finiteness of $X_{A_{N;K},K}$ implies that there is
some large number field $F_A$, which can be chosen independently of $L$, such that $\tilde{L}$, and thus $L$, is
contained in $F_A$. So, upon showing the existence of $A_{N;K}$, we
will have shown that $Y_{N,K}$ is finite under the assumption that $K/\Q$ is
Galois.\newline 
\par \textbf{Existence of $A_{N;K}$}: For any $\sigma:L \hookrightarrow \C$, $\sigma$ maps a prime $\mathfrak{P}$
of $L$
ramifying in $L/K$ to a prime $\mathfrak{P}'$ of $\sigma(L)$ ramifying in $\sigma(L)/K$, where
$\mathfrak{P}$ and $\mathfrak{P}'$ lie over the same rational prime $p$. It follows that the rational
primes ramifying in $L/\Q$ are the same as those ramifying in
$\tilde{L}/\Q$, and using the fact that the discriminant is the norm
of the different, one checks that $rd(L)=rd(\sigma(L))$. The degree $[L:\Q]$ is the number of embeddings $\sigma:L \hookrightarrow
\C$. For each such $\sigma$, $\tau \sigma (L)=\sigma(L)$ for all embeddings
$\tau : \sigma(L) \to \C$ that fix $K$. There are
$[\sigma(L):K]=[L:K]$ such $\tau$, so the number of
distinct fields $\sigma(L)$ as $\sigma$ ranges over all embeddings $L
\hookrightarrow \C$ is at most $[L:\Q]/[L:K]=[K:\Q]$. 

\begin{lemma}\label{mult}
Let $E$ and $F$ be number fields, both Galois over $E\cap F$. Then
\begin{align*}
rd(EF)rd(E\cap F)\leq rd(E)rd(F).
\end{align*} 
\end{lemma}
 
\begin{proof}
For notational ease, put $K=E\cap F$. Using \eqref{(1)} from Section
\ref{intro} multiple times, we have
\begin{align*}
 rd(F)&=N_{K/\Q}(d_{F/K})^{\frac{1}{[F:\Q]}}rd(K), \textrm{      and
  }\\
 rd(EF)& =(N_{K/\Q}d_{EF/K})^{\frac{1}{[EF:\Q]}}rd(K) \\
&=N_{K/\Q}\Big(\big(N_{E/K}d_{EF/E}\big)\big(d_{E/K}^{[EF:E]}\big)\Big)^{\frac{1}{[EF:\Q]}}rd(K)\\
&=N_{E/\Q}(d_{EF/E})^{\frac{1}{[EF:\Q]}}N_{K/\Q}(d_{E/K})^{\frac{1}{[E:\Q]}}rd(K)\\
&=N_{E/\Q}(d_{EF/E})^{\frac{1}{[EF:\Q]}}rd(E).
\end{align*}
Thus we must show $N_{E/\Q}(d_{EF/E})^{\frac{1}{[EF:\Q]}}\leq
N_{K/\Q}(d_{F/K})^{\frac{1}{[F:\Q]}}$. 
\begin{figure}
  \setlength{\unitlength}{2.5cm}
  \begin{picture}(1,1.8)(0,-.6)
    \put(1.99,-.5){$\Q$}
    \put(1.83,.03){$E\cap F=K$}
    \put(1.5,.5){$E$}
    \put(2.5,.5){$F$}
    \put(1.95,1){$EF$}
    \qbezier(2.04,-.34)(2.04,-.34)(2.04,-.03)
    \qbezier(1.96,.18)(1.96,.18)(1.62,.44)
    \qbezier(2.14, .18)(2.14,.18)(2.48,.44)
    \qbezier(1.62,.63)(1.62,.63)(1.96,.96)
    \qbezier(2.46,.63)(2.46, .63)(2.14, .96)
  \end{picture}
  \begin{picture}(1,1.8)(-1.3,-.6)
    \put(1.5,.5){$\mathfrak{P}$}
    \put(2.5,.5){$\q$}
    \put(1.99,-.5){$p$}
    \put(2,1){$\mathfrak{Q}$}
       
    \put(2.01,.06){$\p$}
    \qbezier(2.04,-.37)(2.04,-.37)(2.04,-.03)
    \qbezier(1.96,.18)(1.96,.18)(1.62,.44)
    \qbezier(2.14, .18)(2.14,.18)(2.48,.44)
    \qbezier(1.62,.63)(1.62,.63)(1.96,.96)
    \qbezier(2.46,.63)(2.46, .63)(2.14, .96)
  \end{picture}

\end{figure}
 
\par 
Let $\p$ be a prime of $K$ above $p$. It suffices to show that the
exponent $a$ of $\p$ in $d_{F/K}^{[EF:F]}$ is greater than or
equal to the exponent $b$ of $\p$ in $N_{E/K}(d_{EF/E})$. Let $\q$ be a
prime of $F$ above $\p$ and $\mathfrak{Q}$ a prime of $EF$ above
$\q$. Let $\mathfrak{P}$ be the prime of $E$ below $\mathfrak{Q}$.

It follows from the definition of ramification groups that the natural
restriction map $\Eva(EF_{\mathfrak{Q}}/E_{\mathfrak{P}})\to
\Eva(F_{\q}/K_{\p})$ takes $G_{(\mathfrak{Q}/\mathfrak{P}),i}$
(injectively) into $G_{(\q/\p),\lceil \frac{i+1}{e}-1 \rceil}$, where
$\lceil n \rceil$ denotes the least integer greater than or equal to
$n$, and $e=e(\mathfrak{Q}/\q)$. From this we find that
\begin{align*}
\alpha := \sum_{i=0}(|G_{(\mathfrak{Q}/\mathfrak{P}),i}|-1) \leq e
\sum_{i=0}(|G_{(\q/\p),i}|-1):= e\gamma . 
\end{align*}
One checks that 
\begin{align*}
a= \gamma g_{F/K}(\p)f_{F/K}(\p)[EF:F] \vspace{.4cm} \textrm{ and }
\vspace{.4cm} b=f_{E/K}(\p)g_{E/K}(\p)
g_{EF/E}(\mathfrak{P})f_{EF/E}(\mathfrak{P}) \alpha. 
\end{align*}
We can rewrite $b$ as 
\begin{align*}
\alpha f_{EF/F}(\q)g_{EF/F}(\q)f_{F/K}(\p)g_{F/K}(\p).
\end{align*}
Using $[EF:F]=ef_{EF/F}(\q)g_{EF/F}(\q)$ and $\alpha \leq e
\gamma$, we find that $b \leq a$, as desired. This complete the proof
of the lemma. \end{proof}

 For our purposes, we only need the weaker result $rd(EF)\leq
rd(E)rd(F)$. Returning to our previous notation, by Lemma \ref{mult} and induction on $n$ we obtain
\begin{align*}
rd(\sigma_1(L)\cdots \sigma_n(L))\leq \prod_{i=1}^nrd(\sigma(L_i))=rd(L)^n.
\end{align*}
Thus $rd(\tilde{L})\leq rd(L)^{[K:\Q]}$, and we may take
$A_{N;K}=N^{[K:\Q]}$. 
\newline

\par Thus, as long as $K/\Q$ is Galois, we have shown that $Y_{N,K}$
is finite. We now show that the assumption
$K/\Q$ Galois is unnecessary. Recall the following theorem from
algebraic number theory (see, e.g., \cite{MR1697859} or \cite{MR554237}).

\begin{theorem}\label{Ore}
Let $E/F$ be a finite Galois extension of number fields of degree $n$ and let $\p$ be a
prime of $F$ lying below a prime $\mathfrak{P}$ of $E$ with ramification
index $e=e(\mathfrak{P}/\p)$. Then
\begin{align*}
v_{\p}(d_{E/F})\leq n(1+v_{\p}(e)-1/e).
\end{align*}

\end{theorem}

\par Let $\tilde{K}$ denote the Galois closure of $K/\Q$ and suppose
$L/K$ is abelian with $rd(L)\leq N$. We
show there is a constant $C_{N,K}$ such that
\begin{align*}
 rd(L\tilde{K})\leq C_{N,K}rd(L). 
\end{align*}

 Suppose we know this to be true. Since
$\tilde{K}/\Q$ is Galois, there exists a large number field
$F_{C_{N,K}N,\tilde{K}}$ such that $E\subseteq F_{C_{N,K}N,\tilde{K}}$ for any $E \in
Y_{C_{N,K}N,\tilde{K}}$. In particular, since $L\tilde{K}/\tilde{K}$ is
abelian, $L\tilde{K} \in Y_{C_{N,K}N,\tilde{K}}$. Thus
\begin{align*}
 rd(L)\leq N \implies rd(L\tilde{K})\leq C_{N,K}N \implies
 L\subseteq L\tilde{K}\subseteq F_{C_{N,K}N, \tilde{K}} \implies Y_{N,K} \textrm{ is
   finite }.
\end{align*}
\par We show such a $C_{N,K}$ exists. It follows from \eqref{(1)} that 
\begin{align}
rd(L\tilde{K})=\big(N_{L/\Q}d_{L\tilde{K}/L}\big)^{\frac{1}{[L\tilde{K}:\Q]}}rd(L). \label{rdk} 
\end{align}

\par The primes ramifying in $\tilde{K}/K$ are fixed with $K$, so
there exists a finite set of rational primes $R$, independent of $L$, such that every prime of $L$ ramifying
in $L\tilde{K}/L$ lies above a prime in $R$. Let $\q$ be a
prime of $L$ ramifying in $L\tilde{K}/L$ lying above a rational prime
$q$.
\begin{figure}[h]
  \setlength{\unitlength}{2.5cm}
  \begin{picture}(2,1.6)(-1.2,-.5)
    \put(1.97,.03){$K$}
    \put(1.5,.5){$L$}
    \put(2.5,.5){$\tilde{K}$}
    \put(1.95,1){$L\tilde{K}$}
    \put(1.97,-.5){$\Q$}
    \put(1.75, -.5){$q$}
    \put(1.25 ,.5){$\q$}
    \qbezier(1.77, -.4)(1.77, -.4)(1.35,.4)
    \qbezier(2.04,-.34)(2.04,-.34)(2.04,-.03)
    \qbezier(1.96,.18)(1.96,.18)(1.62,.44)
    \qbezier(2.14, .18)(2.14,.18)(2.48,.44)
    \qbezier(1.62,.63)(1.62,.63)(1.90,.96)
    \qbezier(2.46,.63)(2.46, .63)(2.23, .96)
  \end{picture}
 
\end{figure}

Let $e=e_{L\tilde{K}}(\q)$ (note $L\tilde{K}/L$ is Galois). By Theorem
\ref{Ore}, the exponent of $\q$ in $d_{L\tilde{K}/L}$ is at most 
\begin{align*}
[L\tilde{K}:L]\big(1+v_{\q}(e)-1/e\big)\leq [L\tilde{K}:L]\big(1+v_{\q}(e)\big).
\end{align*}
We have the analogous bound for every prime $\q$ above $q$ in $L$. Let
$e_{\q}=e_{L/\Q}(\q)$ and $f_{\q}=f_{L/\Q}(\q)$ (do not confuse $e$
with $e_\q$). We obtain:
\begin{align*}
v_q(N_{L/\Q}d_{L\tilde{K}/L})&\leq
\sum_{\q|q}[L\tilde{K}:L]f_{\q}(v_{\q}(e)+1)\\
&=[L\tilde{K}:L]\sum_{\q|q} f_{\q}(e_{\q}v_q(e)+1)\\
&\leq [L\tilde{K}:L][L:\Q][L\tilde{K}:L],
\end{align*}
where we have used the fact that $\sum f_{\q}e_{\q}=[L:\Q]$ and the obvious
inequality $v_q(e)+1\leq [L\tilde{K}:L]$ to obtain the final
inequality above. In light of \eqref{rdk}, taking 
\begin{align*}
C_{N,K}=\prod_{q \in R}q^{[\tilde{K}:K]}
\end{align*}
gives the result. 
\newline
\par When $n=1$, this completes the proof of Theorem \ref{abelian} in
its entirety. We obtain the Theorem for general $n$ by induction:

 Let $L \in Y_{n, N, K}$. Let $G^{(i)}$ denote the $i$th derived subgroup of $G$, so $n$ is the smallest integer with
$G^{n}=1$. The quotient $G/G^{n-1}$ corresponds to an
intermediate field $L_0$ of $L/K$ with
$G/G^{n-1}\cong \Eva(L_0/K)$. $L_0\subseteq L$ implies $rd(L_0)\leq
rd (L)$, so $L_0 \in Y_{n-1, N, K}$. By induction there are only a finite number of
possibilities for $L_0$. For each such $L_0$, $\Eva(L/L_0)=
G^{n-1}$ is abelian, so by the $n=1$ case of Theorem \ref{abelian}
applied to $L_0$ and $N$, there are only finitely many possibilities for $L$. 
$\qed$

\begin{remark}
\emph{
The proof of the Theorem \ref{abelian} does not lend itself to an
effective bound of the size of $Y_{n,N,K}$ in terms of $n$, $N$ and $K$. }
\end{remark}

\begin{corollary}Fix $N > 0 $, a positive integer $n$, and a number field $K$. Then 
\begin{align*}
\#\{L: \emph{\textit{L/K is solvable,   $\Eva(L/K)\subseteq \GL_n(\F)$
      and  $rd(L) \leq N$}}\}
\end{align*}
is finite, where $\F$ is a finite field that is allowed to vary with $L$. 
\end{corollary}
\begin{proof}
It is a theorem of Zassenhaus \cite{supr} that the length of any solvable subgroup
of $GL_n(\F)$ has a finite bound depending on $n$, independent of
$\F$ (in fact $\F$ need not even be finite). The result now follows
from Theorem \ref{abelian}.
\end{proof} 

\section{Further Questions}
As mentioned in the introduction, a natural extension of Theorem
\ref{abelian} is to consider the size of various subsets
of $Z_{N,K}$. Kedlaya \cite{sqfreedisc}, generalizing work by Yamamoto
\cite{yamamoto}, has shown that for any $n$, there exist
infinitely many real quadratic number fields $K$ admitting an unramified degree $n$
extension $L$ with $\Eva(\tilde{L}/K)$ isomorphic to the alternating
group $A_n$, where $\tilde{L}$ denotes the Galois closure of
$L/K$. It is unknown (at least to the author), however, whether or not a fixed (say real quadratic)
number field may admit an unramified degree $n$ extension with Galois closure
having Galois group $A_n$ for infinitely many $n$. 
\par Fix a number field $K$. Suppose that $Z_{N,K}$ is infinite. One can ask whether
$Z_{N,K}$ must contain an infinite class field tower. Maire \cite{MR1741025} has proven
the existence of infinite unramified extensions of number fields with
class number one. Let $K$ be such a number field with infinite degree
unramified extension $L$. Maire constructs $L$ by constructing an
infinite class field tower of a finite unramified extension of $K$.
This leaves open the question of whether a general infinite unramified
extension $L/K$ must contain an infinite
class field tower--i.e., whether $\Eva(L/K)$ must have a pro-solvable
subquotient. One can also ask the same question but replace the condition $L/K$
unramified with the condition $rd(L)<N$ (meaning $rd(M)<N$ for every field $M$
between $L$ and $K$ with
$M/K$ finite). The answers to such questions are likely beyond the scope of
class field theory. 
\par Since there are only finitely many number fields of bounded root
discriminant of any fixed degree, finding infinite subsets of $Z_{N,K}$
amounts to finding $L \in Z_{N,K}$ with arbitrarily large
degree. One may alternatively consider how many number fields $L$ of
degree $n$ over $K$ with $rd(L)<N$ exist, with $n$ fixed and $N$
varying. Significant work in this direction has been done by Ellenberg and
Venkatesh \cite{MR2199231} and others. 

\begin{section}{Acknowledgments}
I would like to thank my advisor, Joe Silverman, for suggesting the
question of the finiteness of $Y_{N,K}$ and for his helpful
conversations. I would also like to thank the referee for his/her comments.

\end{section}

\bibliographystyle{plain}
\bibliography{myrefs}

\end{document}